\documentclass[11pt,a4paper]{article}

\usepackage[margin=1in]{geometry}
\usepackage{amsmath,amssymb,amsthm,mathtools}
\usepackage{enumitem}
\usepackage{microtype}
\usepackage{cite}
\usepackage{authblk}
\usepackage[
  colorlinks=true,
  linkcolor=blue,
  citecolor=black,
  urlcolor=blue
]{hyperref}

\allowdisplaybreaks
\newtheorem{theorem}{Theorem}[section]
\newtheorem{lemma}[theorem]{Lemma}

\newtheorem{claim}{Claim}
\theoremstyle{remark}

\newcommand{\V}{V}
\newcommand{\E}{E}
\newcommand{\N}{N}
\newcommand{\dd}{d}
\newcommand{\kk}{\kappa}
\newcommand{\ks}{\kappa^{*}}
\newcommand{\ceil}[1]{\left\lceil #1\right\rceil}
\newcommand{\floor}[1]{\left\lfloor #1\right\rfloor}

\hypersetup{
  pdftitle={An Improved Degree Condition for Connectivity-Preserving Spanning (u,v)-Paths},
  pdfauthor={Heng Yang}
}

\title{\bf An Improved Degree Condition for\\
Connectivity-Preserving Spanning $(u,v)$-Paths}
\author{Heng Yang}
\affil{Department of Mathematics, East China Normal University, Shanghai, 200241, China}
\date{August 19, 2026}

\begin{document}
\maketitle
\begin{abstract}
Teng and Tian~\cite{TengTian} proved the following result.  Let $k\ge2$ and
$t\ge3$, and let $G$ be a $k$-connected graph of order $n$.  If
$n\ge6k+1$ and $\delta(G)\ge\ceil{(n+6)/2}$ when $t=3$, while
$n\ge6k+7t-17$ and $\delta(G)\ge\ceil{(n+t+2)/2}$ when $t\ge4$, then, for
any two distinct vertices $u,v$ and every integer $s$ with $1\le s\le t$,
there exist $s$ internally vertex-disjoint $(u,v)$-paths $P_1,\dots,P_s$ whose
union spans $G$ and such that $G-\E(P_1\cup\cdots\cup P_s)$ is $k$-connected.
They asked whether the minimum-degree condition could be lowered to
$\delta(G)\ge\ceil{(n+t)/2}$ for every $t\ge3$.  We answer this question
affirmatively and further reduce the required order to
$n\ge\max\{6k+9-3t,\;2k+t+3\}$.


\end{abstract}

\section{Introduction}

Throughout this paper, all graphs are finite, simple and undirected.  For
terminology and notation not defined here, we follow West~\cite{West}.  For a
graph $G$, let $\V(G)$ and $\E(G)$ denote its vertex set and edge set,
respectively, and call $|\V(G)|$ the order of $G$.  Let $\kk(G)$,
$\delta(G)$ and $\Delta(G)$ denote the connectivity, minimum degree and
maximum degree of $G$, respectively.  We say that $G$ is $k$-connected if
$\kk(G)\ge k$.  For $v\in\V(G)$, let $\N_G(v)$ denote its neighborhood and
let $\dd_G(v):=|\N_G(v)|$ be its degree.  For $X\subseteq\V(G)$, let $G[X]$
denote the subgraph of $G$ induced by $X$, and write
$G-X:=G[\V(G)\setminus X]$.  For $F\subseteq\E(G)$, write $G-F$ for the
graph obtained from $G$ by deleting the edges of $F$.

Finding a prescribed substructure whose removal leaves the graph
highly connected is a classical theme in connectivity theory.  In 1972,
Chartrand, Kaugars and Lick proved the following well-known result.

\begin{theorem}[Chartrand, Kaugars and Lick~\cite{ChartrandKaugarsLick}]
Every $k$-connected graph $G$ with
$\delta(G)\ge\floor{3k/2}$ has a vertex $x$ such that $G-x$ is
$k$-connected.
\end{theorem}

Subsequent work considered the simultaneous removal of larger connected
vertex sets.  In particular, Fujita and Kawarabayashi obtained the following
two-vertex extension.

\begin{theorem}[Fujita and Kawarabayashi~\cite{FujitaKawarabayashi}]
Every $k$-connected graph $G$ with
$\delta(G)\ge\floor{3k/2}+2$ contains two adjacent vertices $x,y$ such that
$G-\{x,y\}$ is $k$-connected.
\end{theorem}

The corresponding edge-deletion problem asks whether the edges of a prescribed
subgraph can be removed while preserving the connectivity of the whole graph.
Related results on preserving connectivity after deleting the edges of a tree
can be found in~\cite{HasunumaTrees,LiuLiuHong,YangTian}.
Hamiltonian cycles are of particular interest in this context.  A classical
theorem of Dirac states that every graph $G$ of order $n\ge3$
with $\delta(G)\ge n/2$ contains a Hamiltonian cycle.  Such a graph is commonly
called a \emph{Dirac graph}.  For Dirac graphs, Hasunuma proved the following
connectivity-preserving result.

\begin{theorem}[Hasunuma~\cite{Hasunuma}]\label{thm:Hasunuma}
For $k\ge2$, every $k$-connected graph $G$ of order $n\ge6k+10$ with
$\delta(G)\ge n/2$ contains a Hamiltonian cycle $C$ such that
$G-\E(C)$ is $k$-connected.
\end{theorem}

For distinct vertices $u,v\in\V(G)$, a collection of internally
vertex-disjoint $(u,v)$-paths is called a \emph{spanning $(u,v)$-path system}
if the union of their vertex sets is $\V(G)$.  The \emph{spanning
connectivity} $\ks(G)$ is the largest integer $q$ such that, for every
$1\le s\le q$ and every pair of distinct vertices $u,v\in\V(G)$, the graph
$G$ contains a spanning system of $s$ internally vertex-disjoint
$(u,v)$-paths.  The case $s=1$ is a Hamiltonian $(u,v)$-path, while the case
$s=2$ is equivalent to a Hamiltonian cycle.  Thus spanning path systems place
these two familiar Hamiltonian structures in a common framework.

For Hamiltonian paths, Teng and Tian obtained the following result.

\begin{theorem}[Teng and Tian~\cite{TengTian}]\label{thm:ham-path}
Let $k\ge2$, and let $G$ be a $k$-connected graph of order $n\ge6k+6$ with
$\delta(G)\ge\ceil{(n+1)/2}$.  Then, for every pair of distinct vertices
$u,v\in\V(G)$, there exists a Hamiltonian $(u,v)$-path $P$ such that
$G-\E(P)$ is $k$-connected.
\end{theorem}

For spanning $(u,v)$-path systems, Teng and Tian also obtained the following
result.

\begin{theorem}[Teng and Tian~\cite{TengTian}]\label{thm:TengTian}
Let $k\ge2$ and $t\ge3$, and let $G$ be a $k$-connected graph of order $n$
satisfying
\[
\begin{cases}
 n\ge 6k+1,\quad
 \delta(G)\ge \ceil{\dfrac{n+6}{2}}, & t=3,\\[2mm]
 n\ge 6k+7t-17,\quad
 \delta(G)\ge \ceil{\dfrac{n+t+2}{2}}, & t\ge4.
\end{cases}
\]
Then, for every pair of distinct vertices $u,v\in\V(G)$ and every integer $s$
with $1\le s\le t$, there exist $s$ internally vertex-disjoint
$(u,v)$-paths $P_1,\dots,P_s$ whose union spans $G$ and such that
$G-\E(P_1\cup\cdots\cup P_s)$ is $k$-connected.
\end{theorem}

At the end of their paper, Teng and Tian~\cite{TengTian} also asked whether
the degree condition in Theorem~\ref{thm:TengTian} could be replaced by
$\delta(G)\ge\ceil{(n+t)/2}$ for every $t\ge3$.

In this paper, we answer this question affirmatively while also improving the
order condition.

\begin{theorem}\label{thm:main}
Let $k\ge2$ and $t\ge3$, and let $G$ be a $k$-connected graph of order $n$
satisfying
\[
 \begin{cases}
  n\ge6k+9-3t, & 3\le t\le k+1,\\[1mm]
  n\ge2k+t+3, & t\ge k+2,
 \end{cases}
\]
and
\begin{equation}\label{eq:degree-main}
 \delta(G)\ge\ceil{\frac{n+t}{2}}.
\end{equation}
Then, for every pair of distinct vertices $u,v\in\V(G)$ and every integer
$s$ with $1\le s\le t$, there exist $s$ internally vertex-disjoint
$(u,v)$-paths $P_1,\dots,P_s$ such that
$\V(P_1\cup\cdots\cup P_s)=\V(G)$ and
$\kk\!\left(G-\E(P_1\cup\cdots\cup P_s)\right)\ge k$.
\end{theorem}

Since
\[
 6k+9-3t\ge2k+t+3
 \quad\Longleftrightarrow\quad t\le k+1,
\]
the order hypothesis is equivalent to
\begin{equation}\label{eq:order-main}
 n\ge\max\{6k+9-3t,\;2k+t+3\}.
\end{equation}
For convenience, we use either equivalent form as needed below.

The paper is organized as follows.  Section~2 presents preliminary results on
spanning connectivity and Hamiltonicity.  The proof of our main result is
given in Section~3.  Section~4 concludes the paper.

\section{Preliminaries}

For a graph $G$ and a nonnegative integer $j$, write
$\psi_j(G):=\bigl|\{x\in\V(G):\dd_G(x)\le j\}\bigr|$.

We first recall the spanning-connectivity bound used to obtain an initial
path system.

\begin{lemma}[Lin, Huang and Hsu~\cite{LinHuangHsu}]\label{lem:spanning}
Let $G$ be a graph of order $n$ satisfying
$\frac{n}{2}+1\le\delta(G)\le n-2$.  Then
$\ks(G)\ge2\delta(G)-n+2$.
\end{lemma}

A graph is \emph{Hamiltonian-connected} if every pair of distinct vertices is
joined by a Hamiltonian path.  We use the following degree-sequence
criterion.

\begin{lemma}[Chartrand, Kapoor and Kronk~\cite{ChartrandKapoorKronk}]
\label{lem:HC}
Let $G$ be a graph of order $n\ge4$.  If $\psi_j(G)<j-1$ for every
integer $j$ with $2\le j\le\floor{n/2}$, then $G$ is
Hamiltonian-connected.
\end{lemma}

The following classical criterion is the key to reserving a sparse
subgraph.

\begin{lemma}[P\'osa~\cite{Posa}]\label{lem:Posa}
Let $G$ be a graph of order $n\ge3$.  If $\psi_j(G)<j$ for every integer
$j$ with $1\le j<n/2$, then $G$ is Hamiltonian.
\end{lemma}

We also use the set form of Menger's theorem: in a $k$-connected graph, for
any two disjoint vertex sets $A$ and $B$ with $|A|,|B|\ge k$, there are $k$
pairwise vertex-disjoint $A$--$B$ paths; see~\cite{Diestel}.

The last preliminary lemma replaces the uniform minimum-degree step in the
original proof.  It records that only a small number of vertices can lose two
edges when the reserved subgraph is deleted.

\begin{lemma}\label{lem:reserve}
Let $k\ge2$ and $t\ge3$.  Let $G$ be a graph of order $n$ satisfying
\eqref{eq:degree-main} and \eqref{eq:order-main}.  Let $B$ be a subgraph
of $G$ such that
\[
 \Delta(B)\le2,
 \qquad Z:=\{x\in\V(G):\dd_B(x)=2\},
 \qquad |Z|\le k+1.
\]
Set $H:=G-\E(B)$.  Then the following statements hold.
\begin{enumerate}[label=\textup{(\roman*)},leftmargin=2.8em]
\item\label{reserve:i}
If $R\subseteq\V(G)$ and $|R|\le t-2$, then $H-R$ is Hamiltonian.

\item\label{reserve:ii}
If $u,v\notin R$, $uv\in\E(H)$, $|R|\le t-3$, and
$\dd_B(u),\dd_B(v)\in\{0,2\}$, then $(H-R)-uv$ is Hamiltonian.

\item\label{reserve:iii}
Suppose that $u,v\notin R$, $uv\in\E(H)$,
\[
 2\delta(G)-n=t,\qquad \dd_B(u)=\dd_B(v)=2,\qquad |R|=t-2.
\]
Then $(H-R)-uv$ is Hamiltonian.
\end{enumerate}
\end{lemma}

\begin{proof}
Put $N:=n-|R|$ and $q:=\floor{N/2}$.  Since $|R|\le t-2$ and
$n\ge2k+t+3$, we have $N=n-|R|\ge n-t+2\ge2k+5$.
Thus
\begin{equation}\label{eq:qbound}
 \begin{cases}
  q\ge k+3, & N\text{ even},\\
  q\ge k+2, & N\text{ odd}.
 \end{cases}
\end{equation}

For~\ref{reserve:i}, every $x\in\V(H-R)$ satisfies
\begin{equation}\label{eq:reserve-gap}
 \dd_{H-R}(x)\ge \delta(G)-|R|-\dd_B(x),
 \qquad
 2(\delta(G)-|R|)-N=2\delta(G)-n-|R|\ge t-|R|\ge2.
\end{equation}
If $N=2q$, then $\delta(G)-|R|\ge q+1$.  Every vertex of $Z$ has degree at
least $q-1$ in $H-R$, and every vertex outside $Z$ has degree at least
$q$.  Hence
\[
 \psi_j(H-R)=0\quad(1\le j\le q-2),
 \qquad
 \psi_{q-1}(H-R)\le k+1<q-1.
\]
If $N=2q+1$, then $2(\delta(G)-|R|)-N$ is odd and, by
\eqref{eq:reserve-gap}, at least two.  Hence
$2(\delta(G)-|R|)-N\ge3$, so $\delta(G)-|R|\ge q+2$.  Thus
\[
 \psi_j(H-R)=0\quad(1\le j\le q-1),
 \qquad
 \psi_q(H-R)\le k+1<q.
\]
In both cases Lemma~\ref{lem:Posa} applies.

For~\ref{reserve:ii}, put $F=(H-R)-uv$.  Now
$2(\delta(G)-|R|)-N\ge t-|R|\ge3$.  If $N=2q$, the left-hand side is
even and therefore at least four, so $\delta(G)-|R|\ge q+2$; only $u$
and $v$ can have degree at most $q-1$ in $F$.  Hence
\[
 \psi_j(F)=0\quad(1\le j\le q-2),
 \qquad \psi_{q-1}(F)\le2<q-1.
\]
If $N=2q+1$, the only vertices of degree at most $q-1$ are $u,v$,
and every vertex of degree at most $q$ belongs to $Z$.  Therefore
\[
 \psi_j(F)=0\quad(1\le j\le q-2),
 \quad \psi_{q-1}(F)\le2<q-1,
 \quad \psi_q(F)\le k+1<q.
\]
Again Lemma~\ref{lem:Posa} applies.

For~\ref{reserve:iii}, since $|R|=t-2$
and $2\delta(G)-n=t$,
we have $2(\delta(G)-|R|)-N=2$, so $N=2q$ and
$\delta(G)-|R|=q+1$.  In
$F=(H-R)-uv$, each vertex of $Z\setminus\{u,v\}$ has degree at least
$q-1$, whereas $u,v$ have degree at least $q-2$.  Hence
\[
 \psi_j(F)=0\quad(1\le j\le q-3),
 \quad \psi_{q-2}(F)\le2<q-2,
 \quad \psi_{q-1}(F)\le k+1<q-1,
\]
where the strict inequalities follow from $q\ge k+3$.  A final
application of Lemma~\ref{lem:Posa} completes the proof.
\end{proof}

\section{Proof of Theorem~\ref{thm:main}}

\begin{proof}
Fix distinct vertices $u,v\in\V(G)$ and an integer $s$ with
$1\le s\le t$.

\medskip
\noindent
\textbf{Case 1: $\delta(G)=n-1$.}

Then $G=K_n$.  Since $n\ge2k+t+3$, we have
$n-2\ge s$.  Partition $\V(G)\setminus\{u,v\}$ into $s$ nonempty sets
$X_1,\dots,X_s$.  For each $i$, choose a
$(u,v)$-path $P_i$ whose internal vertex set is exactly $X_i$.
Then $P_1,\dots,P_s$ form a spanning system of $s$ internally
vertex-disjoint $(u,v)$-paths.  Put $P:=P_1\cup\cdots\cup P_s$.

Let $F:=G-\E(P)$.  We show directly that $F$ is $k$-connected.  Let
$X\subseteq\V(G)$ with $|X|\le k-1$.  Every vertex
$x\notin\{u,v\}$ has $\dd_F(x)=n-3$.
Suppose that $F-X$ has two distinct components $C_1$ and $C_2$, each
containing a vertex outside $\{u,v\}$.  For $i\in\{1,2\}$, choose
$x_i\in \V(C_i)\setminus\{u,v\}$.  Since $x_i$ has no neighbor in
$F-X$ outside $C_i$, we have
\[
 |\V(C_i)|
 \ge \dd_{F-X}(x_i)+1
 \ge \dd_F(x_i)-|X|+1
 =n-|X|-2.
\]
Consequently,
\[
 n-|X|=|\V(F-X)|
 \ge |\V(C_1)|+|\V(C_2)|
 \ge 2(n-|X|-2).
\]
This implies $n-|X|\le4$.  On the other hand,
$n-|X|\ge2k+t+3-(k-1)=k+t+4>4$, a contradiction.  Hence all vertices
of $\V(G)\setminus(X\cup\{u,v\})$ belong to one component $C$ of $F-X$.

It remains to show that every vertex of $\{u,v\}\setminus X$ also
belongs to $C$.  Let $z\in\{u,v\}\setminus X$.  We have
\[
 \dd_{F-X}(z)
 \ge n-1-s-|X|
 \ge n-t-k
 \ge k+3.
\]
If $z\notin\V(C)$, then the component of $F-X$ containing $z$ has at
most two vertices, both belonging to $\{u,v\}$.  Hence
$\dd_{F-X}(z)\le1$, contrary to $\dd_{F-X}(z)\ge k+3$.  Thus
$z\in\V(C)$.  Therefore $C$
contains every vertex of $F-X$, and so $F-X$ is connected.  Since $X$
was arbitrary, $F$ is $k$-connected.

\medskip
\noindent
\textbf{Case 2: $\delta(G)\le n-2$.}

By \eqref{eq:degree-main}, $\delta(G)\ge n/2+1$.
Lemma~\ref{lem:spanning} yields
\[
 \ks(G)\ge2\delta(G)-n+2\ge t+2.
\]
Hence there is a spanning system
$\mathcal P_0=\{P_1^0,\dots,P_s^0\}$ of $s$ internally vertex-disjoint
$(u,v)$-paths.  Put $P_0:=P_1^0\cup\cdots\cup P_s^0$.
If $G-\E(P_0)$ is $k$-connected, take $P:=P_0$.  Assume that
$G-\E(P_0)$ is not $k$-connected, and set $G':=G-\E(P_0)$ and
$k':=\kk(G')<k$.  Choose $W\subseteq\V(G')$ with $|W|=k'$ such that
$G'-W$ is disconnected.

\begin{claim}\label{clm:components}
The graph $G'-W$ has exactly two components, say $G_1$ and $G_2$.
Moreover, for $i=1,2$,
\begin{equation}\label{eq:component-bounds}
 \delta(G)-k'-1\le|\V(G_i)|\le n-\delta(G)+1.
\end{equation}
\end{claim}

\begin{proof}[Proof of Claim~\ref{clm:components}]
Let $G_1,\dots,G_r$ be the components of $G'-W$.  Since
$\Delta(P_0)\le t$, every vertex $x\in\V(G_i)$ satisfies
\[
 \dd_{G'}(x)\ge \delta(G)-t.
\]
Therefore
\[
 |\V(G_i)|\ge \delta(G)-t-k'+1.
\]
We first verify that the right-hand side is at least four.  If
$3\le t\le k+1$, then
\[
 \delta(G)\ge\ceil{(6k+9-3t+t)/2}=3k+5-t\ge k+t+2,
\]
whereas if $t\ge k+2$, then
\[
 \delta(G)\ge\ceil{(2k+t+3+t)/2}=k+t+2.
\]
Since $k'\le k-1$, it follows that
\[
 \delta(G)-t-k'+1\ge4.
\]
Thus every $G_i$ contains a vertex
$x_i\notin\{u,v\}$.
Such a vertex is internal to exactly one path in $\mathcal P_0$, so
$\dd_{P_0}(x_i)=2$.
Consequently
\begin{equation}\label{eq:component-lower}
 |\V(G_i)|\ge \delta(G)-k'-1.
\end{equation}
If $r\ge3$, then
\[
 n-k'\ge3(\delta(G)-k'-1).
\]
Hence
\[
 3\delta(G)\le n+2k'+3\le n+2k+1.
\]

On the other hand, we have
\[
 3\delta(G)\ge n+3k+5.
\]
Indeed, if $n+t$ is odd, then
\[
 3\delta(G)\ge\frac{3(n+t+1)}2\ge n+3k+5,
\]
by $n\ge6k+9-3t$.  If $n+t$ is even, equality
$n=6k+9-3t$ is impossible because $6k+9-2t$ is odd; hence
$n\ge6k+10-3t$, and
\[
 3\delta(G)\ge\frac{3(n+t)}2\ge n+3k+5.
\]
This contradiction shows that $r=2$.
Applying
\eqref{eq:component-lower} to the component opposite
$G_i$ gives
\[
 |\V(G_i)|=n-k'-|\V(G_{3-i})|
       \le n-k'-(\delta(G)-k'-1)=n-\delta(G)+1.
\]
This proves the claim.
\end{proof}

\begin{claim}\label{clm:reserve-paths}
There are $k$ pairwise vertex-disjoint paths $Q_1,\dots,Q_k$ joining
$G_1$ to $G_2$ such that their ends avoid $u,v$.  If
\[
 B:=Q_1\cup\cdots\cup Q_k,
 \qquad Z:=\{x\in\V(G):\dd_B(x)=2\},
\]
then
\begin{equation}\label{eq:B-properties}
 \Delta(B)\le2,
 \qquad |Z|\le k+1,
 \qquad \dd_B(u),\dd_B(v)\in\{0,2\}.
\end{equation}
\end{claim}

\begin{proof}[Proof of Claim~\ref{clm:reserve-paths}]
For $i=1,2$, set $A_i:=\V(G_i)\setminus\{u,v\}$.
By
\eqref{eq:component-bounds},
\[
 |A_i|\ge \delta(G)-k'-3\ge \delta(G)-k-2.
\]
If $3\le t\le k+1$, then
\[
 \delta(G)\ge3k+5-t\ge2k+4,
\]
while if $t\ge k+2$, then
\[
 \delta(G)\ge k+t+2\ge2k+4.
\]
Consequently $|A_i|\ge k+2$.
Since $G$ is $k$-connected, Menger's theorem gives $k$ pairwise
vertex-disjoint $A_1$--$A_2$ paths.  From each path take a minimal
subpath whose first vertex lies in $A_1$ and whose last vertex lies in
$A_2$.  Denote the resulting paths by $Q_1,\dots,Q_k$.

Their internal vertices avoid $A_1\cup A_2$.  Since
$\V(G)\setminus(A_1\cup A_2)\subseteq W\cup\{u,v\}$, every degree-two
vertex of $B$ belongs to $W\cup\{u,v\}$.  The paths are vertex-disjoint,
so $\Delta(B)\le2$ and $|Z|\le |W|+2=k'+2\le k+1$.
The ends of every $Q_i$ lie in $A_1$ and $A_2$, so neither $u$ nor
$v$ can have degree one in $B$.  Thus
$\dd_B(u),\dd_B(v)\in\{0,2\}$.
\end{proof}

\begin{claim}\label{clm:new-system}
The graph $H:=G-\E(B)$ contains a spanning system
$\mathcal P=\{P_1,\dots,P_s\}$ of $s$ internally vertex-disjoint
$(u,v)$-paths.
\end{claim}

\begin{proof}[Proof of Claim~\ref{clm:new-system}]
We first show that $H$ is Hamiltonian-connected.  By
\eqref{eq:B-properties}, $\delta(H)\ge \delta(G)-2\ge\floor{n/2}$,
whereas every vertex outside $Z$ has degree at least
$\delta(G)-1\ge\floor{n/2}+1$.  Put $q:=\floor{n/2}$.  Since
$n\ge2k+t+3\ge2k+6$, we have $q\ge k+3$.  Therefore
\[
 \psi_j(H)=0\quad(2\le j\le q-1),
 \qquad
 \psi_q(H)\le|Z|\le k+1<q-1.
\]
Lemma~\ref{lem:HC} implies that $H$ is Hamiltonian-connected.  This
settles the case $s=1$.  For $s=2$, choose an edge $xy\in\E(H)$.  Since
$|\V(H)|\ge3$, a Hamiltonian $x$--$y$ path in $H$ does not use $xy$;
adding $xy$ therefore forms a Hamiltonian cycle.  Its two $u$--$v$
arcs give the required two paths.

Assume from now on that $3\le s\le t$.

\medskip
\noindent
\emph{Subcase 3.1: $uv\notin\E(H)$.}

Then
\begin{align*}
 |\N_H(u)\cap\N_H(v)|
 &\ge \dd_H(u)+\dd_H(v)-(n-2)\\
 &\ge2\delta(G)-n+2-\dd_B(u)-\dd_B(v)\\
 &\ge t-2.
\end{align*}
Choose $R\subseteq\N_H(u)\cap\N_H(v)$ with $|R|=s-2$.
Lemma~\ref{lem:reserve}\ref{reserve:i} gives a Hamiltonian cycle in
$H-R$.  Its two $u$--$v$ arcs, together with the paths $uxv$ for all
$x\in R$, form the required spanning system.

\medskip
\noindent
\emph{Subcase 3.2: $uv\in\E(H)$.}

Then
\begin{equation}\label{eq:common}
 |\N_H(u)\cap\N_H(v)|
 \ge2\delta(G)-n-\dd_B(u)-\dd_B(v)\ge t-4.
\end{equation}
If $|\N_H(u)\cap\N_H(v)|\ge s-3$, choose
$R\subseteq\N_H(u)\cap\N_H(v)$ with $|R|=s-3$.
By
Lemma~\ref{lem:reserve}\ref{reserve:ii}, $(H-R)-uv$ has a
Hamiltonian cycle.  Its two $u$--$v$ arcs, the edge $uv$, and the
paths $uxv$ for $x\in R$ give the desired $s$ paths.

Otherwise, \eqref{eq:common} and $s\le t$ give
\[
 t-4\le|\N_H(u)\cap\N_H(v)|\le s-4\le t-4.
\]
Hence
\[
 s=t,
 \qquad |\N_H(u)\cap\N_H(v)|=t-4.
\]
Equality in \eqref{eq:common} forces
\begin{equation}\label{eq:exception-equality}
                  2\delta(G)-n=t,
       \qquad \dd_B(u)=\dd_B(v)=2.
\end{equation}
Because $uv\in\E(G)$,
$|\N_G(u)\cap\N_G(v)|\ge2\delta(G)-n=t$.  Deleting $\E(B)$ can remove
at most $\dd_B(u)+\dd_B(v)=4$ vertices from the common neighborhood of
$u$ and $v$.  Hence equality holds throughout, and each of the four
edges of $B$ incident with $u$ or $v$ removes a distinct common neighbor.
Consequently, without loss of generality, we may choose a vertex
$z\in\N_G(u)\cap\N_G(v)$ such that $uz\in\E(B)$ and $vz\in\E(H)$.
Since $uz\notin\E(H)$,
$|\N_H(u)\cap\N_H(z)|\ge2(\delta(G)-2)-(n-2)=t-2$.
Choose $R\subseteq\N_H(u)\cap\N_H(v)$ with $|R|=t-4$, and then choose
$a\in(\N_H(u)\cap\N_H(z))\setminus(R\cup\{v\})$.
The path $uazv$ is internally disjoint from all $uxv$, $x\in R$.
Put $R':=R\cup\{a,z\}$.  Then $|R'|=t-2$, and
Lemma~\ref{lem:reserve}\ref{reserve:iii} gives a Hamiltonian cycle in
$(H-R')-uv$.  Its two $u$--$v$ arcs, the edge $uv$, the path $uazv$,
and the $t-4$ paths $uxv$ form exactly $t=s$ spanning paths.
\end{proof}

Put $P:=P_1\cup\cdots\cup P_s$ and $G^*:=G-\E(P)$.
Since $\E(P)\subseteq\E(H)$ and $H=G-\E(B)$,
\begin{equation}\label{eq:disjoint}
                         \E(P)\cap\E(B)=\varnothing.
\end{equation}

\begin{claim}\label{clm:connectivity}
The graph $G^*$ is $k$-connected.
\end{claim}

\begin{proof}[Proof of Claim~\ref{clm:connectivity}]
Let $X\subseteq\V(G)$ with $|X|\le k-1$, and put
$X_i:=X\cap\V(G_i)$ for $i=1,2$.
For
$x\in\V(G_i)\setminus\{u,v\}$, every edge from $x$ to
$G_{3-i}$ belongs to $P_0$, because $G_1,G_2$ are components of
$(G-\E(P_0))-W$.  Thus there are at most two such edges, and $x$ has
at most $k'$ neighbors in $W$.  The new path system removes at most
two further edges incident with $x$ inside $G_i$.  Hence
\[
 \dd_{G^*[\V(G_i)]}(x)
 \ge \delta(G)-k'-4\ge \delta(G)-k-3.
\]
We now derive the numerical estimate needed at this point.  If $n+t$
is odd, then, using $n\ge6k+9-3t$,
\[
 3\delta(G)\ge\frac{3(n+t+1)}2\ge n+3k+5.
\]
If $n+t$ is even, then $n=6k+9-3t$ is impossible because
$6k+9-2t$ is odd.  Thus $n\ge6k+10-3t$, and again
\[
 3\delta(G)\ge\frac{3(n+t)}2\ge n+3k+5.
\]
Together with the upper bound in Claim~\ref{clm:components}, this gives
\[
 \delta(G)-k-3\ge\frac{|\V(G_i)|+k-2}{2};
\]
indeed, this is equivalent to
\[
 3\delta(G)\ge n+3k+5.
\]
Consequently
\begin{equation}\label{eq:side-degree}
 \dd_{G^*[\V(G_i)]}(x)
 \ge\frac{|\V(G_i)|+k-2}{2}
 \quad(x\in\V(G_i)\setminus\{u,v\}).
\end{equation}

We claim that all vertices of
$\V(G_i)\setminus(X_i\cup\{u,v\})$ lie in one component $C_i$ of
$G^*[\V(G_i)]-X_i$.  Otherwise two components $C,D$ both contain
vertices outside $\{u,v\}$.  By \eqref{eq:side-degree},
\[
 |C|,|D|\ge\frac{|\V(G_i)|+k}{2}-|X_i|.
\]
Therefore
\[
 |\V(G_i)|-|X_i|
 \ge |C|+|D|
 \ge |\V(G_i)|+k-2|X_i|,
\]
which implies $|X_i|\ge k$, a contradiction.  The component $C_i$ is
nonempty because
$|A_i|\ge k+2$ and $|X_i|\le k-1$.

The paths $Q_1,\dots,Q_k$ are vertex-disjoint.  Since
$|X|\le k-1$, at least one of them, say $Q_j$, avoids $X$.  Its ends
belong to $C_1$ and $C_2$.  By \eqref{eq:disjoint}, the whole path
$Q_j$ survives in $G^*$.  Thus $C_1$ and $C_2$ lie in one component,
denoted by $C$, of $G^*-X$.

Let $w\in W\setminus(X\cup\{u,v\})$.  Exactly two edges of $P$ are
incident with $w$, and $w$ has at most $k'-1$ neighbors in $W$.
Therefore
\[
 |\N_{G^*}(w)\cap(\V(G_1)\cup\V(G_2))|
 \ge \delta(G)-k'-1\ge \delta(G)-k.
\]
After excluding $X$ and $u,v$, at least $\delta(G)-2k-1$ neighbors
remain.  If $3\le t\le k+1$, then
$\delta(G)\ge3k+5-t\ge2k+4$, whereas if $t\ge k+2$, then
$\delta(G)\ge k+t+2\ge2k+4$.  Hence $\delta(G)-2k-1\ge3$.
These neighbors lie in $C$, so every
vertex of $W\setminus(X\cup\{u,v\})$ belongs to $C$.

At this stage every vertex of $G^*-X$ except possibly $u,v$ lies in
$C$.  For $z\in\{u,v\}\setminus X$,
\[
 \dd_{G^*-X}(z)
 \ge \delta(G)-s-(k-1)
 \ge \delta(G)-k-t+1.
\]
If $3\le t\le k+1$, then
$\delta(G)\ge3k+5-t\ge k+t+2$,
whereas if $t\ge k+2$, then
$\delta(G)\ge k+t+2$.  Therefore $\delta(G)-k-t+1\ge3$.
Outside $C$ there is at most the other vertex of $\{u,v\}$, so $z$
has a neighbor in $C$.  Thus $G^*-X$ is connected.  Since $X$ was
arbitrary, $\kk(G^*)\ge k$.
\end{proof}

Claim~\ref{clm:connectivity} completes the proof of the theorem.
\end{proof}

\section{Concluding remarks}

Theorem~\ref{thm:main} shows that in sufficiently large $k$-connected graphs,
for all distinct vertices $u,v$, the degree condition $\delta(G)\ge\ceil{(n+t)/2}$
guarantees spanning $(u,v)$-path systems whose edge deletion preserves
$k$-connectivity.  This gives an affirmative answer to the question of Teng
and Tian~\cite{TengTian} and also improves their order hypothesis.  The
resulting order bound $n\ge\max\{6k+9-3t,\;2k+t+3\}$ cannot be lowered using
the estimates in the present proof; any further improvement would require
sharper estimates or a different argument.

\medskip
\noindent\textbf{Declaration of AI Use.}
ChatGPT was used to assist in developing and checking the constructions
and proofs.  The author independently verified all mathematical arguments,
wrote the paper, and takes full responsibility for its content.

\end{document}